\documentclass[11pt]{amsart}
\usepackage[colorlinks,citecolor=blue,urlcolor=black,linkcolor=black]{hyperref}
\usepackage{amssymb, latexsym, enumitem}

\newtheorem{thm}{Theorem}[section]
\newtheorem{lemma}[thm]{Lemma}
\newtheorem{cor}[thm]{Corollary}
\newtheorem{claim}{Claim}[thm]

\newtheorem{fact}[thm]{Fact}
\newtheorem{mainthm}{Theorem}

\theoremstyle{definition}
\newtheorem{defn}[thm]{Definition}
\newtheorem{conv}[thm]{Convention}

\theoremstyle{remark}
\newtheorem{remark}[thm]{Remark}
\newtheorem{rmk}{Remark}

\DeclareMathOperator{\suc}{succ}
\DeclareMathOperator{\cf}{cf}
\DeclareMathOperator{\dom}{dom}
\DeclareMathOperator{\im}{Im}
\DeclareMathOperator{\otp}{otp}
\DeclareMathOperator{\acc}{acc}
\DeclareMathOperator{\nacc}{nacc}
\DeclareMathOperator{\p}{P}
\DeclareMathOperator{\add}{Add}

\newcommand\symdiff{\mathbin\bigtriangleup}
\newcommand\s{\subseteq}
\newcommand\sq{\sqsubseteq}
\newcommand\sqleft[1]{\mathrel{_{#1}{\sqsubseteq}}}

\newcommand\ch{\textup{CH}}
\newcommand\extend{\textsf{extend}}
\newcommand\pvec{\fbox{${\cdots}$}\hspace{1pt}}
\renewcommand\mid{\mathrel{|}\allowbreak}
\renewcommand\restriction{\mathbin\upharpoonright}

\author{Yair Hayut}
\address{Einstein Institute of Mathematics, The Hebrew University of Jerusalem, Israel}
\urladdr{https://math.huji.ac.il/~yairhayut/}
\email{yair.hayut@mail.huji.ac.il}

\author{Assaf Rinot}
\address{Department of Mathematics, Bar-Ilan University, Ramat-Gan 52900, Israel}
\urladdr{https://www.assafrinot.com}

\author{Zhixing You}
\address{Einstein Institute of Mathematics, The Hebrew University of Jerusalem, Israel}
\email{zhixingy121@gmail.com}

\date{Preprint as of \today. For updates, visit \textsf{http://p.assafrinot.com/76}.}

\title{More notions of forcing add a square}

\begin{document}
\begin{abstract} Foreman and Magidor showed that $\ch$ implies the existence of a countably-closed $\aleph_2$-cc forcing notion $\mathbb P$ for adding $\square_{\omega_1}$.
Here, we show that $\mathbb P$ may consistently be realized as an $\aleph_2$-Souslin tree.
More generally, we prove that $\square_\lambda$ may be added by a $\lambda^+$-Souslin tree,
providing the first analog of the Foreman--Magidor forcing at the level of successors of singular cardinals.
Our construction is uniform and extends to inaccessible cardinals as well.
\end{abstract}

\maketitle
\section{Introduction}
The square principle $\square_\lambda$ was introduced by Jensen \cite{jensen} in his study of the fine structure of G\"odel's constructible universe.
He also devised a countably-closed $(\lambda+1)$-strategically-closed forcing notion for adding $\square_\lambda$.
Baumgartner introduced a weakening $\square^B_\lambda$ of $\square_\lambda$ that may be added by a $({<}\lambda)$-directed-closed forcing notion,
hence is compatible with $\lambda$ being supercompact (unlike $\square_\lambda$).
At the level of a regular $\lambda>\aleph_1$,
Baumgartner's square is just one club-shooting away from Jensen's square,
and at the level of $\lambda=\aleph_1$, the two principles are logically equivalent.
More forcing notions for adding $\square_{\aleph_1}$ were introduced by
Dolinar--D\v{z}amonja \cite{MR2974841}, Krueger \cite{MR3194418},
and Neeman \cite{MR3696074}.
A third square principle is Todor{\v{c}}evi{\'c}'s $\square(\kappa)$ \cite{TodActa}
that follows from $\square_\lambda$ whenever $\kappa=\lambda^+$.

It is an unpublished result of Foreman and Magidor that for an infinite regular cardinal $\lambda=\lambda^{<\lambda}$ there is a $\lambda$-closed $\lambda^+$-cc forcing notion that adds Baumgartner's square $\square^B_\lambda$.
Alternative posets for introducing $\square^B_\lambda$ are given in \cite[\S1.3]{MR701126} and \cite[Definition~3.12]{paper31}.
In this paper, we show that squares can be added by a poset as good as a Souslin tree. To exemplify:

\begin{mainthm}\label{thma} For every uncountable cardinal $\lambda=\lambda^{<\lambda}$,
if $\diamondsuit(E^{\lambda^+}_\lambda)$ holds,\footnote{Here, $E^\kappa_\lambda$ stands for $\{\alpha<\kappa\mid\cf(\alpha)=\lambda\}$; $E^\kappa_{\neq\lambda}$, $E^\kappa_{<\lambda}$, $E^\kappa_{>\lambda}$, $E^\kappa_{\ge\lambda}$ are defined similarly.}
then:
\begin{itemize}
\item there is an $\aleph_1$-complete $\lambda^+$-Souslin tree that forces $\square_\lambda$;
\item there is a $\lambda$-complete $\lambda^+$-Souslin tree that forces $\square_\lambda^B$.
\end{itemize}
\end{mainthm}
\begin{rmk} This is sharp: by \cite[Theorem~1]{MR763902} and \cite[Theorem~6.1]{MR2050172},
the Proper Forcing Axiom implies that $\square(\aleph_3)$ cannot be introduced by an $\aleph_2$-closed forcing. By \cite[Corollary~1.16]{paper20},
it is compatible with the existence of a $\cf(\lambda)$-complete $\lambda^+$-Souslin tree for every cardinal $\lambda\ge\aleph_2$.
\end{rmk}

Our tree constructions make use of the proxy principles,
and this has well-known advantages \cite{paper65} that are apparent here as well.
For starters, it provides the first consistent way to add $\square_\lambda$
by a $\lambda^+$-cc forcing notion for $\lambda$ a singular cardinal (indeed, by forcing with $\lambda^+$-Souslin tree). To exemplify one exotic scenario:

\begin{mainthm}\label{thmb} Suppose that $\lambda$ is a measurable cardinal and $2^\lambda=\lambda^+$.
In the forcing extension by Prikry forcing there is a $\lambda^+$-Souslin tree that forces $\square_\lambda$.
\end{mainthm}
\begin{rmk} The preceding model is a two-step iteration of $\lambda^+$-cc notions of forcing.
Note that by \cite[Theorem~11.1]{MR1838355}, $\square_\lambda$ can fail in the intermediate model.
\end{rmk}

Second, it provides analogous results for inaccessibles. To exemplify:
\begin{mainthm}\label{thmc} Suppose that $\diamondsuit$ holds over a nonreflecting stationary subset of a strongly inaccessible $\kappa$.
Then there is a $\kappa$-Souslin tree that forces $\square(\kappa)$.
\end{mainthm}

A third advantage, demonstrated here for concreteness at the level of $\aleph_2$
is that it is possible to get an $\aleph_2$-Souslin tree adding a square without assuming diamond on the critical cofinality (as assumed in Theorem~\ref{thma}).
\begin{mainthm}\label{thmd} If $2^{2^{\aleph_0}}=\aleph_2$ and there exists a nonreflecting stationary subset of $E^{\aleph_2}_{\aleph_0}$, then:
\begin{enumerate}[label=\textup{(\arabic*)}]
\item there is an $\aleph_2$-Souslin tree that forces $\square_{\omega_1}$;
\item there is an $\aleph_1$-complete $\aleph_2$-Souslin tree that forces $\square(\omega_2)$.
\end{enumerate}
\end{mainthm}
\begin{rmk}
We refer the reader to Corollary~\ref{cor4} for consequences beyond squares,
and to Section~\ref{sec4} in general for a whole gallery of applications.
\end{rmk}

We conclude the paper with a complementary result.
\begin{mainthm}\label{thme} Assuming the consistency of large cardinals, the conjunction of the following two bullet points is compatible with $\kappa$ being a successor of a regular uncountable,
a successor of a singular, or a strongly inaccessible.
\begin{itemize}
\item There is a $\kappa$-Souslin tree;
\item Every $\kappa$-Souslin tree forces that $\square(\kappa)$ fails.
\end{itemize}
\end{mainthm}

\subsection{Organization of this paper}
In Section~\ref{sec2}, we provide preliminaries on square and proxy principles.

In Section~\ref{sec3}, we prove the main technical result of this paper: constructing a $\kappa$-Souslin tree $\mathbf T$ from an instance $\p_\xi(\kappa,\kappa,\pvec)$ of the proxy principle
such that forcing with $\mathbf T$ adds the corresponding narrow instance $\p_\xi(\kappa,2,\pvec)$.

In Section~\ref{sec4}, we derive various corollaries, including Theorems \ref{thma}---\ref{thmd}.
The reader may want to start off by reading this section, as it motivates the technical result of Section~\ref{sec3}.

In Section~\ref{sec5}, we prove a strong form of Theorem~\ref{thme}.

\subsection{Notation and conventions}
Throughout the paper, $\kappa$ stands for a regular uncountable cardinal.
As they play the role of notions of forcing,
all $\kappa$-Souslin trees in this paper are understood to be \emph{normal trees},
that is, every node in a $\kappa$-Souslin tree must have $\kappa$ many extensions in the tree.\footnote{Generally speaking, for every $\kappa$-Souslin tree
there is an $\varepsilon<\kappa$ such that each of its nodes of height $\ge\varepsilon$ admits $\kappa$-many extensions in the tree. Normality imposes $\varepsilon:=0$.}
For a set of ordinals $C$ and an ordinal $\sigma$, we write $\acc(C):=\{\alpha\in C\mid \sup(C\cap\alpha)=\alpha>0\}$, $\nacc(C):=C\setminus\acc(C)$,
and $\suc_\sigma(C) := \{ \alpha\in C\mid \otp(C\cap\alpha)\text{ is a successor ordinal}\le\sigma\}$.

\section{Square and proxy principles}\label{sec2}
A \emph{$\mathcal C$-sequence} over $\kappa$ is a sequence $\vec{\mathcal C}=\langle\mathcal C_\alpha\mid\alpha<\kappa\rangle$ such that for every $\alpha<\kappa$,
$\mathcal C_\alpha$ is nonempty collection of closed subsets $C$ of $\alpha$ with $\sup(C)=\sup(\alpha)$.
It is said to be \emph{$\xi$-bounded} provided that $\otp(C)\le\xi$ for every $C\in\bigcup_{\alpha<\kappa}\mathcal C_\alpha$.
It is said to be \emph{unthreadable} provided that for every club $D$ in $\kappa$, there is an $\alpha\in\acc(D)$ such that $D\cap\alpha\notin\mathcal C_\alpha$.\footnote{If $\vec{\mathcal C}$ is $\xi$-bounded for some $\xi<\kappa$, then $\vec{\mathcal C}$ is unthreadable.}
It is said to be \emph{$\mathcal R$-coherent} (for any given binary relation $\mathcal R$)
provided that for all $C\in\bigcup_{\alpha<\kappa}\mathcal C_\alpha$ and $\beta\in\acc(C)$,
there is a $\bar C\in\mathcal C_\beta$ with $\bar C\mathrel{\mathcal R}C$.
The main examples of $\mathcal R$ here are the end-extension relation $\sq$ and its variant $\sq_\chi$ that focuses on clubs of high order-type; they are defined as follows:
\begin{itemize}
\item $\bar C\sq C$ iff there is an ordinal $\beta$ such that $\bar C=C\cap\beta$;
\item $\bar C \sq_\chi C$ iff either $\bar C \sqsubseteq C$ or ($\otp(C)<\chi$ and $\nacc(C)$ consists of successor ordinals).
\end{itemize}

\begin{defn}
\begin{itemize}
\item A \emph{$C$-sequence over $\kappa$} is one $\langle C_\alpha\mid\alpha<\kappa\rangle$ for which
$\langle \{C_\alpha\}\mid\alpha<\kappa\rangle$ is a $\mathcal C$-sequence over $\kappa$;
\item Jensen's $\square_\lambda$ asserts the existence of a $\lambda$-bounded $\sq$-coherent $C$-sequence over $\lambda^+$;
\item Baumgartner's $\square^B_\lambda$ asserts the existence of a $\lambda$-bounded $\sq_\lambda$-coherent $C$-sequence over $\lambda^+$;
\item Todor{\v{c}}evi{\'c}'s $\square(\kappa)$ asserts the existence of an unthreadable $\sq$-coherent $C$-sequence over $\kappa$.
Likewise, $\square(\kappa,{<}\mu)$ asserts the existence of an unthreadable $\sq$-coherent $\mathcal C$-sequence $\langle\mathcal C_\alpha\mid\alpha<\kappa\rangle$
such that $|\mathcal C_\alpha|<\mu$ for every $\alpha<\kappa$.
\end{itemize}
\end{defn}
\begin{remark}\label{rem1}
We mentioned in the introduction that $\square_{\aleph_1}$ and $\square^B_{\aleph_1}$ are logically equivalent.
Note that, more generally,
the existence of a $\lambda$-bounded $\sq_{\aleph_1}$-coherent $C$-sequence over $\lambda^+$
implies that $\square_{\lambda}$ holds.
Indeed, given such a sequence $\langle C_\alpha\mid\alpha<\lambda^+\rangle$,
for every $\alpha\in E^{\lambda^+}_\omega$, if there exists a $\beta\in E^{\lambda^+}_{>\omega}$ such that $\alpha\in\acc(C_\beta)$,
then replace $C_\alpha$ by $C_\beta\cap\alpha$, and otherwise, replace $C_\alpha$ by a cofinal subset of $\alpha$ of order-type $\omega$.
The outcome of this systematic modification at ordinals of countable cofinality is a $\square_\lambda$-sequence.
\end{remark}

We now present a special case of the parametrized proxy principle $\p^-(\ldots)$.
For a discussion on the motivations behind these principles and a comparison with the square and diamond principles,
we refer the reader to \cite[\S2]{paper65}.

\begin{defn}[special case of {\cite[Definition~1.5]{paper22}}]
Suppose:
\begin{itemize}
\item $\xi,\sigma,\chi\le\kappa$ are ordinals;
\item $\mu,\theta\le\kappa$ are cardinals;
\item $\mathcal S$ is a family of stationary subsets of $\kappa$.
\end{itemize}
$\p^-_\xi(\kappa,\mu,\sq_\chi,\allowbreak\theta,\mathcal S,\mu,\sigma)$ asserts the existence of a $\xi$-bounded $\sq_\chi$-coherent $\mathcal C$-sequence $\vec{\mathcal C}=\langle\mathcal C_\alpha\mid\alpha<\kappa\rangle$ such that the two hold:
\begin{enumerate}
\item for every $\alpha<\kappa$, $|\mathcal C_\alpha|<\mu$;
\item\label{hittingeqn}
for every sequence $\langle B_i\mid i<\theta\rangle$ of cofinal subsets of $\kappa$,
for every $S\in\mathcal S$, there are stationarily many $\alpha\in S$ such that,
for every $i<\min\{\alpha,\theta\}$,
for every $C\in\mathcal C_\alpha$,
$\sup\{\gamma\in C\mid \suc_\sigma(C\setminus\gamma)\s B_i\}=\alpha$.
\end{enumerate}
\end{defn}
\begin{remark}\label{remark-succ}
In the special case $\sigma=1$, the ending of requirement~\eqref{hittingeqn} is
equivalent to asserting that $\sup(\nacc(C) \cap B_i) = \alpha$.
\end{remark}

\begin{defn} $\p_\xi(\kappa,\mu,\sq_\chi,\allowbreak\theta,\mathcal S,\mu,\sigma)$ asserts that
$\p^-_\xi(\kappa,\mu,\sq_\chi,\allowbreak\theta,\mathcal S,\mu,\sigma)$ and $\diamondsuit(\kappa)$ both hold.
\end{defn}

\begin{conv} We may omit $\xi$ in which case we mean that $\xi:=\kappa$.
\end{conv}

\begin{fact}[\cite{TodActa}]\label{square}
$\square(\kappa,{<}\kappa)$ entails the existence of a $\kappa$-Aronszajn tree.
\end{fact}

\begin{fact}[{\cite[Proposition~2.2]{paper26}}] $\p(\kappa,\kappa,{\sq},1,\{\kappa\},\kappa,1)$ entails the existence of a $\kappa$-Souslin tree.
\end{fact}
\section{The example}\label{sec3}
\subsection{Setup} Write $H_\kappa$ for the collection of all sets of hereditary cardinality less than $\kappa$,
and fix a well-ordering $\lhd_\kappa$ of $H_\kappa$.
The trees we construct here are of the following simple form.
\begin{defn} A $\kappa$-tree $\mathbf T=(T,{<_T})$ is said to be \emph{streamlined} iff all of the following hold:
\begin{itemize}
\item $<_T$ is $\subsetneq$;
\item $T$ is a subset of ${}^{<\kappa}H_\kappa$;\footnote{Together with the previous bullet point this means that $x<_Ty$ iff $x$ is an initial segment of $y$.}
\item for all $t\in T$ and $\alpha<\dom(t)$, $t\restriction\alpha\in T$.
\end{itemize}
\end{defn}

In order to seal antichains in streamlined $\kappa$-trees, we shall make use of the following verbose version of $\diamondsuit(\kappa)$.

\begin{fact}[{\cite[Lemma~2.2]{paper22}}]\label{90}
$\diamondsuit(\kappa)$ is equivalent to the existence of a sequence $\langle \Omega_\beta\mid\beta<\kappa\rangle$ satisfying that
for all $\Omega\s H_\kappa$ and $p\in H_{\kappa^+}$, there exists an elementary submodel $\mathcal M\prec H_{\kappa^+}$ containing $p$,
such that $\mathcal M\cap\kappa\in\kappa$ and $\mathcal M\cap\Omega=\Omega_{\mathcal M\cap\kappa}$.
\end{fact}

The next definition makes use of a fixed sequence $\langle \Omega_\beta\mid\beta<\kappa\rangle$ as in Fact~\ref{90}.
\begin{defn}[canonical extension action]\label{extendaction}
For every $T\in H_\kappa$, denote $\beta(T):=0$ unless there is a $\beta<\kappa$
such that $T\s{}^{\le\beta}H_\kappa$ and $T \nsubseteq {}^{<\beta}H_\kappa$,
in which case, we do the following:
\begin{itemize}
\item let $\beta(T)$ denote the unique $\beta$ such that $T\s{}^{\le\beta}H_\kappa$ and $T \nsubseteq {}^{<\beta}H_\kappa$.
\item let $T_{\beta(T)}:=\{ t\in T\mid \dom(t)=\beta(T)\}$.
\item for all $s\in T$ and $C\s\beta(T)$, if $\sup(C)<\dom(s)$, consider the sets:
\begin{itemize}
\item $Q_0:=\{ t\in T_{\beta(T)}\mid s\s t\}$, and
\item $Q_1:=\{ t\in T_{\beta(T)}\mid \exists r\in\Omega_\beta( s\cup r\s t)\}$.
\end{itemize}
If $Q_1$ is nonempty, then let $\extend(s,T,C):=\min(Q_1, {\lhd_\kappa})$.
If $Q_1$ is empty but $Q_0$ is nonempty, then let $\extend(s,T,C):=\min(Q_0, {\lhd_\kappa})$.
Otherwise, let $\extend(s,T,C):=s$.
\item for all $s\in T$ and $C\s\beta(T)$, if $\sup(C)=\dom(s)$, consider the sets:
\begin{itemize}
\item $Q_0:=\{ t\in T_{\beta(T)}\mid (s{}^\smallfrown\langle C\rangle)\s t\}$, and
\item $Q_1:=\{ t\in T_{\beta(T)}\mid \exists r\in\Omega_\beta( (s{}^\smallfrown\langle C\rangle)\cup r\s t)\}$.
\end{itemize}
If $Q_1$ is nonempty, then let $\extend(s,T,C):=\min(Q_1, {\lhd_\kappa})$.
If $Q_1$ is empty but $Q_0$ is nonempty, then let $\extend(s,T,C):=\min(Q_0, {\lhd_\kappa})$.
Otherwise, let $\extend(s,T,C):=\extend(s,T,\emptyset)$.
\end{itemize}
\end{defn}

\subsection{The construction}
The next theorem is the main result of this paper. On first reading, we recommend focusing on the special case where
$\chi:=\omega$, $\theta:=1$, $\mathcal S:=\{\kappa\}$ and $\sigma:=1$ (bearing in mind Remark~\ref{remark-succ}).
For the purpose of getting a $\lambda^+$-Souslin tree that adds $\square_\lambda$, also take $\xi:=\lambda$ and $\kappa:=\lambda^+$.

\begin{thm}\label{pro1} Suppose:
\begin{itemize}
\item $\chi$ is an infinite cardinal such that $\nu^{<\chi}<\kappa$ for every $\nu<\kappa$;
\item $\xi,\sigma$ are nonzero ordinals $\le\kappa$;
\item $\theta$ is a nonzero cardinal $\le\kappa$;
\item $\mathcal S$ is a nonempty family of stationary $S\s\kappa$ such that $S\setminus E^\kappa_{<\chi}$ is nonstationary;
\item $\p_\xi(\kappa,\mu,\sq_\chi,\theta,\mathcal S,\mu,\sigma)$ holds with $\mu=\kappa$.
\end{itemize}

Then there exists a streamlined $\chi$-complete $\kappa$-Souslin tree $\mathbf T$
such that in the forcing extension by $\mathbf T$,
$\p_\xi(\kappa,\mu,\sq_\chi,\theta,\mathcal S,\mu,\sigma)$ holds with $\mu=2$.
\end{thm}
\begin{proof} Let $\vec{\mathcal C}=\langle \mathcal C_\alpha\mid \alpha<\kappa\rangle$ be a $\p^-_\xi(\kappa,\kappa,\sq_\chi,\theta,\mathcal S,\kappa,\sigma)$-sequence.
As $S\setminus E^\kappa_{<\chi}$ is nonstationary for every $S\in\mathcal S$,
and as $\nu^{<\chi}<\kappa$ for every $\nu<\kappa$, we may assume that
for every $\alpha\in \acc(\kappa)\cap E^\kappa_{<\chi}$,
$\mathcal C_\alpha$ contains a club $C$ in $\alpha$ of order-type $\cf(\alpha)$ such that $\nacc(C)$ consists of successor ordinals.
We may also assume that $\mathcal C_{\alpha+1}=\{\{\alpha\},\{0,\alpha\}\}$ for every $\alpha<\kappa$.
Lastly, we may also assume that $C\setminus\bar\alpha\in\mathcal C_\alpha$
for all $\bar\alpha<\alpha<\kappa$ and $C\in\mathcal C_\alpha$.

For every $\alpha<\kappa$, let $\mathcal C_\alpha^1:=\{ C\in\mathcal C_\alpha\mid\forall\gamma\in\acc(C)\,[C\cap\gamma\in\mathcal C_\gamma]\}$.
Note that for every $C\in\mathcal C_\alpha^1$ and $\beta\in\acc(C)$, we have $C\cap\beta\in\mathcal C_\beta^1$.
Also note that $\mathcal C_\alpha^1=\mathcal C_\alpha$ for every $\alpha\in E^\kappa_{\ge\chi}$.

Next, write $\mathcal C:=\bigcup_{\alpha<\kappa}\mathcal C_\alpha$.
Let $\mathcal T$ be the collection of all functions $t$ such that:
\begin{itemize}
\item $t\in\prod_{\beta<\alpha}\mathcal C_\beta$ for some $\alpha<\kappa$;
\item for every $\beta<\alpha$, for every $\gamma\in\acc(t(\beta))$, $t(\gamma)\sq_\chi t(\beta)$.
\end{itemize}

We shall construct a sequence $\langle T_\alpha\mid\alpha<\kappa\rangle$
such that, for every $\alpha<\kappa$, $T_\alpha$ will constitute the $\alpha^{\text{th}}$-level of our ultimate tree $T$.
We make the following promises:
\begin{enumerate}[label=\textup{(\arabic*)}]
\item\label{promise1} for every $\alpha<\kappa$, $T_\alpha$ will be a nonempty subset of $\mathcal T\cap{}^\alpha\mathcal C$ of size $<\kappa$;
\item\label{promise2} for all $\beta<\alpha<\kappa$ and $t\in T_\alpha$, $t\restriction\beta\in T_\beta$;
\item\label{promise3} for all $\alpha<\kappa$ and $t\in T_\alpha$, there will be a $C\in\mathcal C_\alpha$ with $t{}^\smallfrown\langle C\rangle\in T_{\alpha+1}$;
\item\label{promise4} for every $\alpha\in\acc(\kappa)\cap E^\kappa_{<\chi}$, $T_\alpha=\{t\in{}^\alpha\mathcal C\mid \forall\beta<\alpha\,(t\restriction\beta\in T_\beta)\}$;\footnote{Put differently, the tree we construct is $\chi$-complete.}
\item\label{promise5} for every $\alpha\in E^\kappa_{\ge\chi}$, $T_\alpha = \{ \mathbf b^C_x \mid C \in \mathcal C_\alpha, x \in T_{\min(C)}\}$, where:
\item\label{promise6} for all $\alpha\in\acc(\kappa)$, $C\in\mathcal C_\alpha^1$ and $x\in T_{\min(C)}$,
$\mathbf b_x^C$ will be some distinguished element of $T_\alpha$ satisfying the following three requirements:
\begin{enumerate}[label=\textup{(\roman*)}]
\item\label{cls1} $\mathbf b_x^C\restriction\min(C)=x$;
\item\label{cls2} for every pair $\beta^-<\beta$ of consecutive ordinals in $C$, $$\mathbf b_x^C\restriction\beta=\extend(\mathbf b_x^C\restriction\beta^-,T\restriction(\beta+1),C\cap\beta^-);$$
\item\label{cls3} for every $\gamma\in\acc(C)$, $\mathbf b_x^C(\gamma)=C\cap\gamma$.
\end{enumerate}
\end{enumerate}

\begin{claim} Suppose that $\zeta\le\kappa$ is such that Promises \ref{promise3}--\ref{promise5} and \ref{promise6}\ref{cls1} hold for every $\alpha<\zeta$.
Then $T\restriction\zeta:=\bigcup_{\alpha<\zeta}T_\zeta$ is a \emph{normal} tree,
that is, for all $\beta<\alpha<\zeta$ and $s\in T_\beta$, there is a $t\in T_\alpha$ such that $s\subsetneq t$.
\end{claim}
\begin{proof} Let $\beta<\zeta$ and $s\in T_\beta$. We prove by induction on $\alpha\in[\beta,\zeta)$
the existence of $t\in T_\alpha$ such that $s\subsetneq t$.
The base case is trivial,
the successor case follows from Promise~\ref{promise3},
and the case $\alpha\in\acc(\zeta\setminus\beta)\cap E^\kappa_{<\chi}$ then follows from Promise~\ref{promise5}.
Finally, given $\alpha\in\acc(\zeta\setminus\beta)\cap E^\kappa_{\ge\chi}$, pick $C\in\mathcal C_\alpha=\mathcal C_\alpha^1$.
As $C$ is a club in $\alpha$, we may find some $\bar\alpha\in C$ above $\beta$. By the induction hypothesis, we may find $x\in T_{\bar\alpha}$
such that $s\subsetneq x$. As $\bar C:=C\setminus\bar\alpha$ belongs to $\mathcal C_\alpha$ and $x\in T_{\min(\bar C)}$,
we get from Promise~\ref{promise6}\ref{cls1} that $s\subsetneq x\s\mathbf b_x^{\bar C}\in T_\alpha$.
\end{proof}

We are now ready for the recursive construction.
We start by letting $T_0:=\{\emptyset\}$, noting that $T_0$ is indeed a subset of $\mathcal T$.
Next, for every $\alpha<\kappa$ such that $T_\alpha$ has already been successfully constructed,
we let $$T_{\alpha+1}:=\{ t{}^\smallfrown\langle C\rangle\mid t\in T_\alpha, C\in\mathcal C_\alpha\}\cap\mathcal T.$$

\begin{claim} Promises \ref{promise1}--\ref{promise3} are maintained.
\end{claim}
\begin{proof} If $\alpha=0$, then $T_{\alpha+1}=\{t\}$ for the unique $t\in\mathcal T$ to satisfy $\dom(t)=1$, namely,
$t:\{0\}\rightarrow\mathcal C_0$.
If $\alpha=\beta+1$ is a successor ordinal, then
$T_{\alpha+1}$ is equal to $\{ t{}^\smallfrown\langle\{\beta\}\rangle,t{}^\smallfrown\langle\{0,\beta\}\rangle\mid t\in T_\alpha\}$ which again takes care of Promises \ref{promise1}--\ref{promise3}.

$\blacktriangleright$ If $\alpha\in\acc(\kappa)\cap E^\kappa_{<\chi}$, then given $t\in T_\alpha$,
we may fix $C\in\mathcal C_\alpha$ of order-type less than $\chi$ such that $\nacc(C)$ consists of successor ordinals.
We claim that $t{}^\smallfrown\langle C\rangle$ is in $\mathcal T$, hence in $T_{\alpha+1}$.
This requires that $t(\gamma)\sq_\chi C$ for every $\gamma\in\acc(C)$,
which holds trivially by our choice of $C$.

$\blacktriangleright$ If $\alpha\in\acc(\kappa)\cap E^\kappa_{\ge\chi}$, then given $t\in T_\alpha$,
by Promise~\ref{promise5},
we may fix $C\in\mathcal C_\alpha=\mathcal C_\alpha^1$ and $x\in T_{\min(C)}$ such that $t=\mathbf b_x^C$,
and we claim that $t{}^\smallfrown\langle C\rangle$ is in $\mathcal T$, hence in $T_{\alpha+1}$.
This requires that $t(\gamma)\sq_\chi C$ for every $\gamma\in\acc(C)$,
which is guaranteed by Promise~\ref{promise6}\ref{cls3}.
\end{proof}

Next, suppose that $\alpha\in\acc(\kappa)$ is such that $\langle T_\beta\mid\beta<\alpha\rangle$ has already been successfully defined.
Let $C\in\mathcal C^1_\alpha$ and $x\in T_{\min(C)}$.
We shall obtain $\mathbf b^C_x$ as the limit $\bigcup\im(b^C_x)$ of a $\subsetneq$-increasing sequence $b^C_x\in\prod_{\beta\in C}T_\beta$ canonically obtained by recursion, as follows:
\begin{itemize}
\item $b^C_x(\min(C)):=x$.
\item for every pair $\beta^-<\beta$ of consecutive elements of $C$ such that $b_x^C(\beta^-)$ has already been defined,
let $$b_x^C(\beta):=\extend(b_x^C(\beta^-),T\restriction(\beta+1),C\cap\beta^-).$$
\item for every $\beta\in\acc(C)$ such that $b_x^C\restriction\beta$ has already been defined, let $b^C_x(\beta):=\bigcup\im(b^C_x\restriction\beta)$.
\end{itemize}

We verify that the construction indeed produced a $\subsetneq$-increasing sequence of nodes in $T\restriction\alpha$.
For every pair $\beta^-<\beta$ of consecutive elements of $C$, since $T\restriction(\beta+1)$ is a normal tree thus far,
${\extend(b_x^C(\beta^-),T\restriction(\beta+1),\ldots)}$
provides an extension of $b_x^C(\beta^-)$ belonging to $T_\beta$,
so indeed $b_x^C(\beta^-)\subsetneq b_x^C(\beta)\in T_\beta$.
For every $\beta\in\acc(C)$, Clauses \ref{cls1} and \ref{cls2} of Promise~\ref{promise6} yield that
$b_x^C(\gamma)=\mathbf b_x^{C\cap\beta}\restriction\gamma$ for every $\gamma\in C\cap\beta$,
and hence $b_x^C(\beta)=\mathbf b_x^{C\cap\beta}$ which was placed in $T_{\beta}$ in an earlier stage of the recursion,
since $C\cap\beta\in\mathcal C_\beta^1$.

Next, if $\cf(\alpha)<\chi$, then define $T_\alpha$ according to Promise~\ref{promise4}.
It is trivial to verify that $T_\alpha\s\mathcal T$.
If $\cf(\alpha)\ge\chi$. then define $T_\alpha$ according to Promise~\ref{promise5}.

\begin{claim} Promises \ref{promise1}--\ref{promise6} are maintained.
\end{claim}
\begin{proof} Promise~\ref{promise1} is maintained in case $\cf(\alpha)<\chi$,
since $|T\restriction\alpha|<\kappa$ and then the cardinal arithmetic hypothesis implies $|T_\alpha|<\kappa$.
It is also maintained in case $\cf(\alpha)\ge\chi$ because in this case $|T_\alpha|\le|T\restriction\alpha|\cdot |\mathcal C_\alpha|<\kappa$.
The verification of Promise~\ref{promise2} is immediate.
Promise~\ref{promise3} does not come into play.
So, we are left with verifying Clause~\ref{cls3} of Promise~\ref{promise6}.
Let $C\in\mathcal C^1_\alpha$ and $x\in T_{\min(C)}$,
and we shall prove that $\mathbf b_x^C(\gamma)=C\cap\gamma$ for every $\gamma\in\acc(C)$.
By induction. Specifically, given $\beta\in\acc(C)$ such that $\mathbf b_x^C(\gamma)=C\cap\gamma$ for every $\gamma\in\acc(C\cap\beta)$,
we argue as follows.
Let $\beta^+:=\min(C\setminus\beta+1)$, so that $\mathbf b_x^C(\beta)=b_x^C(\beta^+)(\beta)$.
It is the case that $$b_x^C(\beta^+)=\extend(b_x^C(\beta),T\restriction(\beta^++1),C\cap\beta).$$
We have that $\sup(C\cap\beta)=\beta=\dom(b_x^C(\beta))$, so by the definition of $\extend(\ldots)$,
it suffices to show that $t:=b_x^C(\beta){}^\smallfrown\langle C\cap\beta\rangle$ belongs to $T_{\beta+1}$.
As $b_x^C(\beta)\in T_\beta$, the definition of $T_{\beta+1}$ implies that we need to show that $t$ belongs to $\mathcal T$,
i.e., that for every $\gamma\in\acc(t(\beta))$, $t(\gamma)\sq_\chi t(\beta)$.
But $t(\beta)=C\cap\beta$ and $t(\gamma)=C\cap\gamma$ for every $\gamma\in\acc(C\cap\beta)$,
so we are good.
\end{proof}

Having constructed all levels of the tree, we let
$T:= \bigcup_{\alpha < \kappa} T_\alpha$, so that $\mathbf T:=(T,{\subsetneq})$ is our streamlined $\kappa$-tree.

\begin{claim}\label{c233} $\mathbf T$ admits no antichains of size $\kappa$.
\end{claim}
\begin{proof}
Suppose not, and let $A \subseteq T$ be a maximal antichain of size $\kappa$.
By \cite[Claim~2.2.2]{paper26}, the following set is stationary in $\kappa$:
$$B := \{ \beta \in\acc(\kappa)\mid A\cap(T\restriction\beta)= \Omega_\beta\text{ is a maximal antichain in }T\restriction\beta \}.$$
Using the hitting property of the proxy principle, fix an $\alpha\in E^\kappa_{\ge\chi}$ such that
$\sup (\nacc(C) \cap B) = \alpha$
for every $C \in \mathcal C_\alpha$.
We shall prove that $A \subseteq T \restriction \alpha$.
To this end consider any $z \in T \restriction (\kappa \setminus \alpha)$.
Set $y := z \restriction \alpha$, so that $y \in T_\alpha$ and $y \subseteq z$.
Recalling Promise~\ref{promise5}, pick $C\in\mathcal C_\alpha$ and $x\in T_{\min(C)}$ such that $y = \mathbf b^C_x$.
Fix $\beta \in \nacc(C) \cap B$ with $\dom(x) < \beta < \alpha$.
Denote $\beta^-:=\sup(C\cap\beta)$. Then $\beta^- < \beta$ is a pair of consecutive elements of $C$,
so by Promise~\ref{promise6}\ref{cls2},
$$\mathbf b_x^C\restriction\beta=\extend(\mathbf b_x^C\restriction\beta^-,T\restriction(\beta+1),C\cap\beta^-).$$

Since $\beta \in B$, $\beta$ is a limit ordinal and $\Omega_\beta = A \cap (T \restriction \beta)$ is a maximal antichain in $T \restriction \beta$.
It follows that for every $s'\in T_{\beta^-}\cup T_{\beta^-+1}$ there is an $r\in\Omega_\beta$ such that $s'\cup r\in T\restriction\beta$,
and by the normality of $T\restriction(\beta+1)$, there is then a $t\in T_\beta$ with $s'\cup r\s t$.
This verifies that the definition of $\extend(\mathbf b_x^C\restriction\beta^-,T\restriction(\beta+1),C\cap\beta^-)$
will end up being $\min(Q_1,\lhd_\kappa)$,\footnote{Recall Definition~\ref{extendaction}.}
so that $\mathbf b_x^C\restriction\beta$
extends some $r\in\Omega_\beta=A\cap(T\restriction\beta)$.
Since $r$ is an element of the antichain $A$,
and $r \subseteq b^C_x(\beta) \subsetneq \mathbf b^C_x = y \subseteq z$,
we infer that $z \notin A$.
\end{proof}

\begin{claim} $\mathbf T$ admits no chains of size $\kappa$.
\end{claim}
\begin{proof} Suppose not, and pick $b:\kappa\rightarrow\mathcal C$ such that $b\restriction\alpha\in T_\alpha$ for every $\alpha<\kappa$.
For every infinite $\alpha<\kappa$, as $\mathcal C_{\alpha+1}=\{\{\alpha\},\{0,\alpha\}\}$,
the definition of $T_{\alpha+2}$ implies that
$$t_\alpha:=b\restriction(\alpha+1){}^\smallfrown\langle b(\alpha+1)\symdiff\{0\}\rangle$$
is an element of $T_{\alpha+2}\setminus\{b\restriction(\alpha+2)\}$.
So $\{ t_\alpha\mid \omega\le\alpha<\kappa\}$ is a $\kappa$-sized antichain,
contradicting Claim~\ref{c233}.
\end{proof}

So, $\mathbf T$ is a $\kappa$-Souslin tree.
Next, work in $V[G]$, where $G$ is $\mathbf T$-generic over $V$.
Set $b:=\bigcup G$, so that $b:\kappa\rightarrow\mathcal C$ is a map such that
$b\restriction\alpha\in T_\alpha$ for every $\alpha<\kappa$.
As $T_\alpha\s\mathcal T\cap{}^\alpha\mathcal C$ for every $\alpha<\kappa$,
$b$ is a $\xi$-bounded $\sq_{\chi}$-coherent $C$-sequence over $\kappa$, which we hereafter denote by $\vec C=\langle C_\alpha\mid\alpha<\kappa\rangle$.
The next claim is unnecessary, but its proof is short and may be of interest to those readers whose primary interest is in classical square principles.
\begin{claim}\label{cla1} $\vec C$ is unthreadable.
\end{claim}
\begin{proof}
Given a club $D\s\kappa$, use the $\kappa$-cc of $\mathbf T$ to find a ground model club $D'\s D$.
Working in the ground model, invoke the hitting property of the proxy principle with respect to $B:=\acc(D')$, and find some $\alpha\in E^\kappa_{\ge\chi}$ such that $\sup(\nacc(C)\cap B)=\alpha$ for every $C\in\mathcal C_\alpha$.
Then $\sup(\nacc(b(\alpha))\cap\acc(D'))=\alpha$,
in particular $\sup(\nacc(b(\alpha))\cap\acc(D))=\alpha$,
and hence $\alpha\in\acc(D)$ with $b(\alpha)\neq D\cap\alpha$.
\end{proof}

Our next task is verifying that $\p^-_\xi(\kappa,2,\sq_\chi,\theta,\mathcal S,2,\sigma)$
holds in $V[G]$.

Back in $V$, as $\diamondsuit(\kappa)$ holds and $\mathbf T$ is a $\kappa$-sized $\kappa$-cc forcing,
the beginning of the proof of \cite[Theorem~3.4]{paper26} provides a sequence $\langle Z_\beta\mid\beta<\kappa\rangle\in V$ satisfying that
for every $B\in\mathcal P^{V[G]}(\kappa)$, there exists some $X_B\in\mathcal P^{V}(\kappa)$ such that:
\begin{enumerate}
\item $V\models X_B\text{ is stationary}$;
\item $V[G]\models X_B\s \{\beta<\kappa\mid Z_\beta=B\cap\beta\}$.
\end{enumerate}

In particular, $\diamondsuit(\kappa)$ holds in $V[G]$.
Without loss of generality, we may assume that $Z_\beta=\emptyset$ for every $\beta\in\nacc(\kappa)$.
Working in $V[G]$,
for every $\alpha\in\acc(\kappa)$, let $C^\bullet_\alpha:=\im(g_\alpha)$, where
$g_\alpha:C_\alpha\rightarrow\alpha$ is defined by stipulating:
$$g_\alpha(\beta):=
\begin{cases}
\beta, &\text{if } \beta \in \acc(C_\alpha); \\
\min(Z_\beta\cup\{\beta\}), &\text{if } \beta = \min(C_\alpha); \\
\min\bigl((Z_\beta\cup\{\beta\})\setminus (\sup(C_\alpha \cap \beta)+1)\bigr), &\text{otherwise.}
\end{cases}$$

For completeness, we also set $C_0^\bullet:=\emptyset$ and $C_{\alpha+1}^\bullet:=\{\alpha\}$ for every $\alpha<\kappa$.
By \cite[Lemma~2.8]{paper29}, $\vec C^\bullet:=\langle C_\alpha^\bullet\mid\alpha<\kappa\rangle$ is yet another $\xi$-bounded $\sq_\chi$-coherent $C$-sequence over $\kappa$.
To verify it witnesses $\p^-_\xi(\kappa,2,\sq_\chi,\theta,\mathcal S,2,\sigma)$,
let $\langle B_i\mid i<\theta\rangle$ be a given sequence of cofinal subsets of $\kappa$,
and let $S\in\mathcal S$; we need to find stationarily many $\alpha\in S$ such that
$$\sup\{\gamma\in C^\bullet_\alpha\mid \suc_\sigma(C^\bullet_\alpha\setminus\gamma)\s B_i\}=\alpha.$$

For each $i<\theta$,
fix $X_i\in\mathcal P^{V}(\kappa)$ such that:
\begin{enumerate}[label=\textup{(\arabic*)}]
\item $V\models X_i\text{ is stationary}$;
\item $V[G]\models X_i\s \{\beta<\kappa\mid Z_\beta=B_i\cap\beta\}$.
\end{enumerate}

As any club in $V[G]$ covers a club from $V$, we may shrink the stationary set $X_i$ to also ensure the following:
\begin{enumerate}[label=\textup{(\arabic*)}, resume]
\item $V[G]\models X_i\s \{\beta<\kappa\mid \sup(B_i\cap\beta)=\beta\}$.
\end{enumerate}

Working in $V$, since $S\in\mathcal S$, there is a stationary subset $\bar S$ of $S$ such that for every $\alpha\in\bar S$,
for every $C\in\mathcal C_\alpha$,
for every $i<\min\{\alpha,\theta\}$,
$$\sup\{\gamma\in C\mid \suc_\sigma(C\setminus\gamma)\s X_i\}=\alpha.$$
In particular, for every $i<\min\{\alpha,\theta\}$,
the following set is cofinal in $\alpha$:
$$\Gamma_i:=\{\gamma\in C_\alpha\mid \suc_\sigma(C_\alpha\setminus\gamma)\s X_i\}.$$

Work in $V[G]$. As it is a $\kappa$-cc forcing extension of $V$, $\bar S$ remains stationary.
In addition, a moment's reflection makes it clear that for every $i<\theta$,
for every $\beta\in\nacc(C_\alpha)\cap X_i$, it is the case that $g_\alpha(\beta)\in\nacc(C_\alpha^\bullet)\cap B_i$.\footnote{See the proof of \cite[Claim~3.4.5]{paper26} for details.}
It follows that for every $i<\min\{\alpha,\theta\}$,
for every $\gamma\in\Gamma_i$,
it is the case that $g_\alpha(\gamma)$ is an element of $C^\bullet_\alpha$
satisfying that $\suc_\sigma(C^\bullet_\alpha\setminus\gamma)\s B_i$.
So we are done.
\end{proof}

\section{Corollaries}\label{sec4}
We now present a gallery of corollaries to Theorem~\ref{pro1}.
\begin{cor}\label{cor5}
Suppose $\kappa$ is a regular uncountable cardinal,
and $\p(\kappa,\kappa,\allowbreak{\sq},1,\{\kappa\},\allowbreak\kappa,1)$ holds.
Then there is a $\kappa$-Souslin tree that forces $\square(\kappa)$.
\end{cor}
\begin{proof} By Theorem~\ref{pro1},
there is a streamlined $\kappa$-Souslin tree $\mathbf T$
such that in the forcing extension by $\mathbf T$,
$\p(\kappa,2,{\sq},1,\{\kappa\},2,1)$ holds.
By \cite[Lemma~3.2]{paper22}, then, $\square(\kappa)$ holds in the extension.
\end{proof}

The next corollary establishes Theorem~\ref{thmc}.

\begin{cor} Suppose $\diamondsuit$ holds over a nonreflecting stationary subset of a strongly inaccessible $\kappa$.
Then there is a $\kappa$-Souslin tree that forces $\square(\kappa)$.
\end{cor}
\begin{proof} By \cite[Theorem~4.26]{paper23},
in particular, $\p(\kappa,\kappa,\allowbreak{\sq},1,\{\kappa\},\allowbreak\kappa,1)$ holds.
Now appeal to Corollary~\ref{cor5}.
\end{proof}

\begin{cor}\label{cor3}
Suppose $\lambda$ is an infinite cardinal,
and $\p_\lambda(\lambda^+,\lambda^+,\allowbreak{\sq},1,\{\lambda^+\},\allowbreak\lambda^+,1)$ holds.
Then there is a $\lambda^+$-Souslin tree that forces $\square_\lambda$.
\end{cor}
\begin{proof} By Theorem~\ref{pro1},
there is a streamlined $\lambda^+$-Souslin tree $\mathbf T$
such that in the forcing extension by $\mathbf T$,
$\p_\lambda(\lambda^+,2,{\sq},1,\{\lambda^+\},2,1)$ holds.
In particular, in the extension, there is a $\lambda$-bounded $\sq$-coherent $C$-sequence over $\lambda^+$, i.e., $\square_\lambda$ holds.
\end{proof}

The next corollary establishes Theorem~\ref{thmb}.

\begin{cor} Suppose $\lambda$ is a measurable cardinal satisfying $2^\lambda=\lambda^+$.
In the forcing extension by Prikry forcing
there is a $\lambda^+$-Souslin tree that forces $\square_\lambda$.
\end{cor}
\begin{proof} By \cite[Main Theorem]{paper26}, the hypotheses imply that in the forcing extension by Prikry forcing,
$\p(\lambda^+,\lambda^+,{\sq},1,\{\lambda^+\},\lambda^+,1)$ holds.
Now appeal to Corollary~\ref{cor3}.
\end{proof}

The next two corollaries establish together Theorem~\ref{thma}.

\begin{cor}\label{cor2}
Suppose $\lambda=\lambda^{<\lambda}$ is an infinite cardinal and $\diamondsuit(E^{\lambda^+}_\lambda)$ holds.
Then there is a $\lambda$-complete $\lambda^+$-Souslin tree $\mathbf T$
that forces $\square^B_\lambda$. In particular:
\begin{enumerate}[label=\textup{(\arabic*)}]
\item\label{cls6} If $\lambda=\aleph_1$, then $\mathbf T$ forces $\square_\lambda$;
\item\label{cls7} If $\lambda>\aleph_1$ is a successor cardinal, then $\mathbf T$ also forces the existence of a $\lambda$-complete $\lambda^+$-super-Souslin tree.
\end{enumerate}
\end{cor}
\begin{proof} By \cite[Theorem~5.6]{paper22}, the hypotheses imply that
$\p_\lambda(\lambda^+,2,\allowbreak\mathcal R,\allowbreak\lambda^+,\{E^{\lambda^+}_\lambda\},2,\sigma)$ holds for
$\mathcal R:={\sqleft{\lambda}}$ and every $\sigma<\lambda$, in particular, $\sigma:=1$.
As $\lambda^{<\lambda}=\lambda$, it follows that
$\p_\lambda(\lambda^+,\lambda^+,\mathcal R,\lambda^+,\{E^{\lambda^+}_\lambda\},2,1)$ holds
for $\mathcal R:={\sq}$, in particular, for $\mathcal R:={\sq_\lambda}$.
By Theorem~\ref{pro1}, then, there exists a streamlined $\lambda$-complete $\lambda^+$-Souslin tree $\mathbf T$
such that in the forcing extension by $\mathbf T$,
$\p_\lambda(\lambda^+,2,{\sq_\lambda},\lambda^+,\{E^{\lambda^+}_{\lambda}\},2,1)$ holds.
In particular, in the extension, there is a $\lambda$-bounded $\sq_\lambda$-coherent $C$-sequence over $\lambda^+$, i.e., $\square^B_\lambda$ holds.

Finally, Clause~\ref{cls6} follows from the fact that $\square_{\aleph_1}$ and $\square^B_{\aleph_1}$ are logically equivalent.
Clause~\ref{cls7} follows from Corollary~1.7 of \cite{paper31} together with the sentence prior to Theorem~$B''$ of the same paper.
\end{proof}

\begin{cor} Suppose $\lambda=\lambda^{<\lambda}$ is an uncountable cardinal and $\diamondsuit(E^{\lambda^+}_\lambda)$ holds.
Then there is a countably complete $\lambda^+$-Souslin tree $\mathbf T$
that forces $\square_\lambda$.
\end{cor}
\begin{proof} As explained in the proof of Corollary~\ref{cor2},
$\p_\lambda(\lambda^+,\lambda^+,{\sq},\lambda^+,\{E^{\lambda^+}_\lambda\},\allowbreak2,1)$ holds.
In particular, $\p_\lambda(\lambda^+,\lambda^+,{\sq_{\aleph_1}},\lambda^+,\{E^{\lambda^+}_\lambda\},\allowbreak2,1)$ holds.
By Theorem~\ref{pro1}, then, there exists a countably complete $\lambda^+$-Souslin tree $\mathbf T$
such that in the forcing extension by $\mathbf T$,
$\p_\lambda(\lambda^+,2,{\sq_{\aleph_1}},\lambda^+,\{E^{\lambda^+}_{\lambda}\},2,1)$ holds.
In particular, in the extension, there is a $\lambda$-bounded $\sq_{\aleph_1}$-coherent $C$-sequence over $\lambda^+$.
By Remark~\ref{rem1}, then, $\square_\lambda$ holds.
\end{proof}

The next corollary establishes Theorem~\ref{thmd}.

\begin{cor}\label{cor1}
Suppose $\lambda=2^{<\lambda}<2^\lambda=\lambda^+$ is uncountable cardinal,
and there exists a stationary subset of $E^{\lambda^+}_{\neq\cf(\lambda)}$ that does not reflect.
\begin{enumerate}[label=\textup{(\arabic*)}]
\item\label{cls4} If $\lambda$ is regular or if $\square^*_\lambda$ holds,
then there is a $\lambda^+$-Souslin tree that forces $\square_\lambda$;
\item\label{cls5} If $\lambda$ is regular, then there is a $\lambda$-complete $\lambda^+$-Souslin tree that forces $\square(\lambda^+)$.
\end{enumerate}
\end{cor}
\begin{proof} \ref{cls4} If $\lambda$ is a regular cardinal then the cardinal arithmetic hypothesis implies that $\square^*_\lambda$ holds.
Thus, by \cite[Corollary~3.15]{paper32}, either of the hypothesis of Clause~\ref{cls4} imply that
$\p_\lambda(\lambda^+,\lambda^+,{\sq},{<}\lambda,\allowbreak\{E^{\lambda^+}_\eta\},2,{<}\lambda)$ holds for some infinite regular cardinal $\eta<\lambda$.
In particular,
$\p_\lambda(\lambda^+,\lambda^+,{\sq},1,\{\lambda^+\},\allowbreak\lambda^+,1)$ holds.
By Theorem~\ref{pro1}, then, there exists a streamlined $\lambda^+$-Souslin tree $\mathbf T$
such that in the forcing extension by $\mathbf T$,
$\p_{\lambda}(\lambda^+,1,{\sq},2,\allowbreak\{\lambda^+\},2,1)$ holds.
In particular, in the extension, there is a $\lambda$-bounded $\sq$-coherent $C$-sequence over $\lambda^+$, i.e., $\square_\lambda$ holds.

\ref{cls5}
By \cite[Theorem~A]{paper32}, the hypotheses imply that
$\p(\lambda^+,\lambda^+,\allowbreak\mathcal R,\allowbreak1,\{E^{\lambda^+}_\lambda\},2,1)$ holds for
$\mathcal R:={\sq^*}$.
By \cite[Corollary~4.36]{paper23}, it moreover holds for $\mathcal R:={\sq}$.
By Theorem~\ref{pro1}, then, there exists a streamlined $\lambda$-complete $\lambda^+$-Souslin tree $\mathbf T$
such that in the forcing extension by $\mathbf T$,
$\p(\lambda^+,1,{\sq_\lambda},2,\{E^{\lambda^+}_{\lambda}\},2,1)$ holds.
By \cite[Lemma~3.2]{paper22}, then, $\square(\lambda^+)$ holds in the extension.
\end{proof}

To put the next corollary in context, recall that
Laver proved (see \cite{Sh:104}) that in the generic extension after L\'evy-collapsing a weakly compact cardinal to $\aleph_2$,
there are no $\aleph_2$-Aronszajn trees with an $\omega$-ascent path,
and that Todor{\v{c}}evi{\'c} proved \cite[Theorem~4.4]{MR631563} that in the same generic extension,
no $\aleph_2$-Aronszajn tree is coherent.
\begin{cor}\label{cor4}
Suppose $2^{2^{\aleph_0}}=\aleph_2$ and $E^{\aleph_2}_{\aleph_0}$ admits a nonreflecting stationary subset.
Then there is an $\aleph_1$-complete $\aleph_2$-Souslin tree $\mathbf T$ such that:
\begin{enumerate}[label=\textup{(\arabic*)}]
\item\label{cls8} $\mathbf T$ forces the existence of a uniformly coherent $\aleph_2$-Souslin tree;
\item\label{cls9} $\mathbf T$ forces the existence of an $\aleph_2$-Souslin tree with an $\omega$-ascent path.
\end{enumerate}
\end{cor}
\begin{proof} As seen in the proof of Corollary~\ref{cor1}\ref{cls5},
there exists a streamlined countably complete $\aleph_2$-Souslin tree $\mathbf T$
such that in the forcing extension by $\mathbf T$,
$\p(\aleph_2,2,{\sq_{\aleph_1}},1,\{E^{\aleph_2}_{\aleph_1}\},2,1)$ holds.
Equivalently, in this case,
$\p(\aleph_2,2,\allowbreak{\sq},\allowbreak1,\{E^{\aleph_2}_{\aleph_1}\},2,1)$ holds.
By \cite[Theorem~3.11(3)]{paper28}, moreover,
$\p(\aleph_2,2,\allowbreak{\sq},\allowbreak\aleph_2,\{E^{\aleph_2}_{\aleph_1}\},2,1)$ holds.

\ref{cls8} By \cite[Theorem~6.35]{paper23},
$\p(\aleph_2,2,\allowbreak{\sq},\allowbreak\aleph_2,\{\aleph_2\},2,1)$
is sufficient for the construction of a uniformly coherent $\aleph_2$-Souslin tree in the extension.

\ref{cls9} By \cite[Corollary~6.12]{paper23}, $\square(\aleph_2)$ together with $2^{2^{\aleph_0}}=\aleph_2$
suffices for the construction of $\aleph_2$-Souslin tree with an $\omega$-ascent path.
\end{proof}

\begin{cor} Suppose $\kappa$ is a subtle cardinal and $\p(\kappa,\kappa,{\sq},1,\{\kappa\},\kappa,1)$ holds.
Then there is $\kappa$-Souslin tree that forces the existence of a full $\kappa$-Souslin tree.
\end{cor}
\begin{proof} We start by fixing some notation.
For each $x\in{}^{<\kappa}H_\kappa$, consider the following options:
\begin{itemize}[label=$\blacktriangleright$]
\item If $\im(x)\s {}^2(H_\kappa)$,
then let $x^0,x^1$ be the unique maps from $\dom(x)$ to $H_\kappa$ such that $x(\gamma)=(x^0(\gamma),x^1(\gamma))$ for every $\gamma<\dom(x)$.
\item Otherwise, let $x_\beta^0,x_\beta^1$ be arbitrary maps from $\dom(x)$ to $H_\kappa$.
\end{itemize}

Next, as $\kappa$ is subtle,
by \cite[Proposition~3.7]{paper62}, we may fix a stationary subset $S$ of $\kappa$ consisting of regular uncountable cardinals
such that $\diamondsuit_S^*(\kappa\textup{-trees})$ holds.
This means that we may fix a sequence $\langle x_\beta\mid\beta<\kappa\rangle$ such that
for every streamlined $\kappa$-tree $X$, there are club many $\alpha\in S$ such that for every $x\in X$ of height $\alpha$,
the set $\{\beta<\alpha\mid x\restriction\beta=x_\beta\}$ is stationary in $\alpha$.

\begin{claim}\label{cla2} Suppose $V[G]$ is a forcing extension obtained by forcing with a given streamlined $\kappa$-Souslin tree $T$ in $V$.
Then $\diamondsuit_S^*(\kappa\textup{-trees})$ holds in $V[G]$.
\end{claim}
\begin{proof} First, note that since $T$ is $\kappa$-distributive,
every stationary subset of a cardinal $\alpha\in S$ remains stationary in $V[G]$.
Second, note that the translation procedure of \cite[\S7.1]{paper20} can be used to show that the restriction of the principle $\diamondsuit_S^*(\kappa\textup{-trees})$ to guessing streamlined subtress of ${}^{<\kappa}2$
is no weaker than the full $\diamondsuit_S^*(\kappa\textup{-trees})$. So it suffices to prove that this restricted version holds in $V[G]$.
Letting $y_\beta:=\bigcup\im(x_\beta^1)$ for every $\beta<\kappa$,
we claim that $\langle y_\beta\mid\beta<\kappa\rangle$ is such a witness.

To this end, let $Y$ be a streamlined $\kappa$-subtree of ${}^{<\kappa}2$ in $V[G]$,
and let $\dot Y$ be a $T$-name for it. Back in $V$,
define $X$ to be the collection of all $x\in {}^{<\kappa}({}^2(H_\kappa))$
such that:
\begin{itemize}
\item $x^0\in T$;
\item for every ordinal $\gamma$ such that $\gamma+1<\dom(x)$, $x^0\restriction\gamma$ forces that $x^1(\gamma)$ is an element of $\dot Y\cap{}^{<\gamma}2$.
\end{itemize}

As $T$ is a $\kappa$-Souslin tree, we may fix a club $D$ in $\kappa$ such that for every $\beta\in D$, every node in $T$ of height $\beta$
decides $\dot Y\restriction\beta$.
It follows that, in $V[G]$, for every $\alpha\in\acc(D)$, for every $y\in Y$ of height $\alpha$,
there is an $x\in X$ of height $\alpha$ such that $x^0\in G$ and $\bigcup(\im(x^1_\alpha))=y$.

Back in $V$, as $T$ is a $\kappa$-tree and as $|{}^\gamma2|<\kappa$ for every $\gamma<\kappa$, $X$ is a streamlined $\kappa$-tree,
so we may fix a club $C$ in $\kappa$ such that for every $\alpha\in S\cap C$ and every $x\in X$ of height $\alpha$,
the set $B_x:=\{\beta<\alpha\mid x\restriction\beta=x_\beta\}$ is stationary in $\alpha$.
Work in $V[G]$ and let $\alpha\in S\cap C\cap \acc(D)$.
Given any $y\in Y$ of height $\alpha$,
we may find an $x\in X$ of height $\alpha$ such that $x^0\in G$ and $\bigcup(\im(x^1))=y$.
As $\alpha$ is a regular uncountable cardinal, the set $B_{x,y}:=\{\beta<\alpha\mid \bigcup(\im(x^1\restriction\beta))=y\restriction\beta\}$ is a club in $\alpha$,
so that altogether $B_x\cap B_{x,y}$ is stationary in $\alpha$, and for every $\beta$ in that intersection,
$y_\beta=\bigcup\im(x_\beta^1)=\bigcup\im((x\restriction\beta)^1)=\bigcup\im(x^1\restriction\beta)=y\restriction\beta$.
\end{proof}

Appeal to Theorem~\ref{pro1}
to fix a streamlined $\kappa$-Souslin tree $T$
such that in the forcing extension by $T$,
$\p(\kappa,2,{\sq},1,\{\kappa\},2,1)$ holds.
As well, by Claim~\ref{cla2}, $\diamondsuit_S^*(\kappa\textup{-trees})$ holds in this extension.
So we may invoke \cite[Theorem~4.1]{paper62}.
\end{proof}

\begin{cor} Suppose $\lambda=\lambda^{<\lambda}$ is an uncountable cardinal and $2^\lambda=\lambda^+$.
Then $\add(\lambda,1)$ introduces a $\lambda$-complete $\lambda^+$-Souslin tree that forces the existence of a full $\lambda^+$-Souslin tree.
\end{cor}
\begin{proof} Recall that $V^{\add(\lambda,1)}\models\diamondsuit(\lambda)$.
Meanwhile, by \cite[Theorem~5.7]{paper23}, in $V^{\add(\lambda,1)}$,
$\p_\lambda(\lambda^+,2,\mathcal R,\lambda^+,\{E^{\lambda^+}_\lambda\},\allowbreak\sigma,1)$ holds for
$\mathcal R:={\sqleft{\lambda}}$ and every $\sigma<\lambda$, in particular, $\sigma:=1$.
As $\lambda^{<\lambda}=\lambda$, it follows that (in $V^{\add(\lambda,1)}$)
$\p_\lambda(\lambda^+,\lambda^+,\mathcal R,\lambda^+,\{E^{\lambda^+}_\lambda\},\allowbreak2,1)$ holds
for $\mathcal R:={\sq}$, in particular for $\mathcal R:={\sq_\lambda}$.
By Theorem~\ref{pro1}, then,
in the same model there exists a streamlined $\lambda$-complete $\lambda^+$-Souslin tree $\mathbf T$
such that in the forcing extension by $\mathbf T$,
$\p_{\lambda}(\lambda^+,2,{\sq_\lambda},\lambda^+,\allowbreak\{E^{\lambda^+}_\lambda\},2,1)$ holds.
As $\mathbf T$ is $\lambda^+$-distributive, $\diamondsuit(\lambda)$ remains to hold.
As $\mathbf T$ has size $\lambda^+$, $2^\lambda=\lambda^+$ remains to hold.
Now, appeal to \cite[Theorem~5.1]{paper62}.
\end{proof}

\section{A complementary result}\label{sec5}
In this section, we establish Theorem~\ref{thme}
by modifying an argument of Kunen from {\cite[\S3]{MR495118}}.
Denote by $\mathbb S_{\kappa}$ the forcing notion
consisting of the collection of all normal streamlined homogeneous subtrees of ${}^{<\kappa}2$ that are either empty or of a successor height;
the order on $\mathbb S_{\kappa}$ is defined by taking end extensions.
This forcing adds a $\kappa$-Souslin tree, as follows.

\begin{fact}[{\cite[\S3]{MR495118}}]\label{issouslin}
For any $\mathbb S_{\kappa}$-generic filter $G$, $T(G):=\bigcup G$ is a normal streamlined homogeneous $\kappa$-Souslin tree in $V[G]$.
\end{fact}

For brevity, let us denote $T(G)$ by $T$, and its canonical name by $\dot T$.

\begin{fact}[{\cite[\S3]{MR495118}}]\label{equivalent}
If $\kappa=\kappa^{<\kappa}$, then $\mathbb S_{\kappa}*\dot{T}$ is forcing equivalent to $\add(\kappa,1)$.
\end{fact}

Recall that for a forcing notion $\mathbb P$,
we say that `the tree property holds at $\kappa$ indestructibly under $\mathbb P$'
iff there are no $\kappa$-Aronszajn trees and this remains to be the case in any forcing extension by $\mathbb P$.

\begin{lemma}\label{squarefail} Suppose $\kappa=\kappa^{<\kappa}$,
the tree property holds at $\kappa$ indestructibly under $\add(\kappa,1)$.
Then for any $\mathbb S_{\kappa}$-generic $G$,
the forcing extension $V[G]$ satisfies both of the following:
\begin{enumerate}[label=\textup{(\arabic*)}]
\item\label{leme1} There exists a $\kappa$-Souslin tree;
\item\label{leme2} Every $\kappa$-Souslin tree forces that $\square(\kappa,{<}\mu)$ fails for every $\mu<\kappa$.
\end{enumerate}
\end{lemma}
\begin{proof}
Clause~\ref{leme1} is taken care of by Fact~\ref{issouslin},
so we turn to address Clause~\ref{leme2}. To this end, let $T'$ be an arbitrary $\kappa$-Souslin tree in $V[G]$,
and let $b'$ be $T'$-generic over $V[G]$.
Recall that we denote by $T$ the $\kappa$-tree added by $G$.

\begin{claim}\label{c531} There is a $T$-generic $b$ over $V[G]$ such that $V[G][b']\subseteq V[G][b]$.
\end{claim}
\begin{proof} Suppose not. Fix $p\in T'$ forcing that for any $T'$-generic $b'$ over $V[G]$ with $p\in b'$,
there is no $T$-generic $b$ over $V[G]$ such that $V[G][b']\not\subseteq V[G][b]$.

Let $b$ be any $T$-generic over $V[G]$.
By Fact~\ref{equivalent}, $V[G][b]$ is a forcing extension of $V[G]$ by $\add(\kappa,1)$.
By our assumption, then, there are no $\kappa$-Aronszajn trees in $V[G][b]$.
In particular, the $\kappa$-subtree of $T'$ consisting of all nodes that are comparable with $p$ acquire a $\kappa$-branch $b'$ in $V[G][b]$.
Trivially, $b'$ is a $\kappa$-branch through $T'$ with $p\in b'$.
But $T'$ is a $\kappa$-Souslin in $V[G]$, hence any $\kappa$-branch through it is generic over it, so $b'$ is generic over $V[G]$.
However, $V[G][b'] \subseteq V[G][b]$,
contradicting the fact that $p$ forces that $V[G][b']\not\subseteq V[G][b]$.
\end{proof}

In $V[G][b']$, let $\vec{\mathcal C}=\langle \mathcal C_\alpha\mid\alpha<\kappa\rangle$ be a $\sq$-coherent $\mathcal C$-sequence over $\kappa$
such that $(\sup_{\alpha<\kappa}|\mathcal C_\alpha|)^+<\kappa$;
we need to prove that it admits a thread.

Let $b$ be given by the claim, so that $V[G][b']\subseteq V[G][b]$.
Recall that by Fact~\ref{equivalent},
the tree property holds at $\kappa$ in $V[G][b]$,
so Fact~\ref{square} implies that $\square(\kappa,{<}\kappa)$ fails, in particular, we may fix in $V[G][b]$ a club $D$ threading $\vec{\mathcal C}$.
As $V[G][b]$ is a $\kappa$-cc forcing extension of $V[G]$, we may fix a subclub $D'$ of $D$ lying in $V[G]$.
Working back in its extension $V[G][b']$,
for every $\alpha\in\acc(D')$, it is the case that $D'\cap\alpha\s D\cap\alpha\in\mathcal C_\alpha$.
By \cite[Lemma~2.4]{MR3730566}, then, $\vec{\mathcal C}$ admits a thread.
\end{proof}

We are now ready to derive a strong form of Theorem~\ref{thme}.

\begin{cor} Assuming the consistency of large cardinals, the conjunction of the following two bullet points is compatible with $\kappa$ being a successor of a regular uncountable,
a successor of a singular, or a strongly inaccessible.
\begin{itemize}
\item There is a $\kappa$-Souslin tree;
\item Every $\kappa$-Souslin tree forces that $\square(\kappa,{<}\mu)$ fails for every $\mu<\kappa$.
\end{itemize}
\end{cor}
\begin{proof} \textbf{Inaccessible:}
If there is a weakly compact cardinal $\kappa$,
then by Kunen's original argument \cite[\S3]{MR495118},
there is a forcing extension in which $\kappa$ is a strongly inaccessible having the tree property and it is moreover indestructible under $\add(\kappa,1)$.
Since $\kappa$ is inaccessible, we have $\kappa^{<\kappa}=\kappa$ in this model,
so we are in conditions to appeal to Lemma~\ref{squarefail}.

\textbf{Successor of regular uncountable:}
As proved in \cite[Lemma~4.14]{MR3893286},
starting from a weakly compact cardinal and an infinite cardinal $\mu=\mu^{<\mu}$,
one obtains a model in which the tree property holds at $\kappa:=\mu^{++}$ and it is indestructible under $\add(\kappa,1)$.
Moreover, $\kappa^{<\kappa}=\kappa$ holds in this model.

\textbf{Successor of singular:}
By \cite[Theorem~3.1]{MR1420265},
for every increasing sequence $\langle \lambda_n \mid n < \omega \rangle$ of supercompact cardinals,
the tree property holds at $\kappa:=\lambda^+$ for $\lambda:=\sup_{n<\omega}\lambda_n$.
By a suitable forcing preparation, it is possible to ensure that $\lambda_n$ be a supercompact that is indestructible under $\add(\kappa,1)$ for each $n<\omega$.
Consequently, the tree property holds at $\kappa$ and is indestructible under $\add(\kappa,1)$.
In addition, $\kappa^{<\kappa}=2^\lambda=\lambda^+=\kappa$
by Solovay's theorem that the Singular Cardinal Hypothesis holds above a strongly compact.
\end{proof}

\section*{Acknowledgments}
The first and third authors were partially supported by the Israel Science Foundation (grant agreement 3469/25).
The second author was partially supported by the Israel Science Foundation (grant agreement 203/22).

\end{document}